\documentclass[a4paper,12pt]{article}
\usepackage[top=1.25in,left=1in,right=1in,bottom=1.25in]{geometry}
\usepackage{latexsym}
\usepackage{epsfig}
\usepackage{enumerate,amsmath,amssymb,dsfont,pifont,amsthm}
\usepackage[dvips]{color}
\usepackage{graphicx}
\usepackage[arrow, matrix, curve]{xy}

\theoremstyle{plain}
\newtheorem{theorem}{Theorem}[section]
\newtheorem{corollary}[theorem]{Corollary}
\newtheorem{lemma}[theorem]{Lemma}
\newtheorem{proposition}[theorem]{Proposition}

\theoremstyle{definition}

\theoremstyle{remark}

\newcommand{\bbc}{\mathbb{C}}
\newcommand{\bbr}{\mathbb{R}}

\newcommand{\bbf}{\mathbb{F}}

\newcommand{\bbn}{\mathbb{N}}

\newcommand{\cb}{\mathcal{B}}
\newcommand{\cf}{\mathcal{F}}

\newcommand{\abs}[1]{\left| #1 \right|}

\begin{document}

\allowdisplaybreaks

\title{\bfseries Well-balanced L\'{e}vy Driven Ornstein-Uhlenbeck Processes}

\author{%
    \textsc{Alexander Schnurr and Jeannette H.C. Woerner}%
    \thanks{Lehrstuhl IV, Fakult\"at f\"ur Mathematik, Technische Universit\"at Dortmund,
              D-44227 Dortmund, Germany,
              \texttt{alexander.schnurr@math.tu-dortmund.de, jeannette.woerner@math.tu-dortmund.de}}
    }

\date{\today}

\maketitle
\begin{abstract}
In this paper we introduce the well-balanced L\'{e}vy driven
Ornstein-Uhlenbeck process as a moving average process of the form
$X_t=\int \exp(-\lambda |t-u|)dL_u$. In contrast to L\'{e}vy driven
Ornstein-Uhlenbeck processes the well-balanced form possesses
continuous sample paths and an autocorrelation function which is
decreasing not purely exponential but of the order $\lambda
|u|\exp(-\lambda |u|)$. Furthermore, depending on the size of
$\lambda$ it allows both for positive and negative correlation of
increments. We indicate how the well-balanced
Ornstein-Uhlenbeck process might be used as mean or volatility
process in stochastic volatility models.
\end{abstract}

\emph{MSC 2010:} 60G10, 60E07, 91B24

\emph{Keywords:} semimartingale, Ornstein-Uhlenbeck process, L\'evy process,   infinitely divisible distribution, autocorrelation, financial modelling


\section{Introduction}
Recently moving average processes have attained much attention, both
from the theoretical and application side, since they provide a
large class of processes, only partly belonging to the class of
semimartingales and allowing to model correlation structures
including long-range dependence. The theoretical foundations of
treating moving average processes with driving L\'{e}vy processes
have been provided in Rajput and Rosinski (1989) and recently the
question under which conditions these type of processes are
semimartingales has been considered in Basse and Pedersen (2009). A
special case of L\'{e}vy driven moving average processes are
fractional L\'{e}vy motions (cf. Benassi et.al (2004) and Marquardt
(2006)), where the kernel function of the fractional Brownian motion
is taken, leading to the same correlation structure as fractional
Brownian motion. Bender et.al (2010) derived conditions on the
driving L\'{e}vy process and the exponent of the kernel function
under which the fractional L\'{e}vy motion is a semimartingale. It
turns out that this can only be the case in the long memory setting
and then the process is of finite variation. Barndorff-Nielsen and
Schmiegel (2009) developed the idea of moving average processes
further by introducing a stochastic volatility component leading to
Brownian semi-stationary processes, which are a very promising class
of processes for modelling turbulence. Furthermore, these processes
have also been applied to electricity modelling (cf.
Barndorff-Nielsen et.al (2010)). However, we can also view the
well-known Ornstein-Uhlenbeck process as moving average process,
which due to its simple structure is very popular for modelling mean
reverting data (e.g. Barndorff-Nielsen and Shephard (2001),
Kl\"uppelberg et.al (2009)).

Motivated by this we introduce an exponential kernel $\exp(-\lambda
|t-\cdot|)$, $\lambda >0$ on the whole real line leading to the
well-balanced Ornstein-Uhlenbeck process. We show that this process
is well defined with only assuming a logarithmic
moment on the driving L\'{e}vy process. The process possesses infinitely divisible
marginal distributions and is stationary. In contrast to L\'{e}vy
driven Ornstein-Uhlenbeck processes it possesses continuous sample
paths of finite variation and therefore it is a semimartingale with
respect to any filtration it is adapted to. Furthermore, the
autocorrelation function is decreasing more slowly than the one of
the Ornstein-Uhlenbeck process with same $\lambda$, namely it is of the order
$\lambda |h|\exp(-\lambda |h|)$. In addition the range of the first-order
autocorrelation of the increments is $(-0.5,1)$ in contrast to
$(-0.5,0)$ for the Ornstein-Uhlenbeck process. Positive values are
often associated to long range dependence, but with the
well-balanced Ornstein-Uhlenbeck process we see that this is not
true.

Hence the well-balanced Ornstein-Uhlenbeck process might serve as a
promising mean process in financial models, e.g. as additive
component in stochastic volatility models, since it possesses the
following desirable properties:
\begin{itemize}
\item{ the decay of the autocorrelation function is of the order $\lambda |h|\exp(-\lambda |h|)$,}
\item{ the autocorrelation between increments can be positive and negative, depending on $\lambda$,}
\item{ it is a semimartingale,}
\item{ it has an infinitely divisible distribution.}
\end{itemize}

Furthermore, as in Barndorff-Nielsen and Shephard (2001) the well-balanced Ornstein-Uhlenbeck process might be used as volatility process in stochastic volatility models. Since the relationship of cumulant transforms between price and volatility process only differs by a constant from the Ornstein-Uhlenbeck model, modelling of the marginals stay the same as in Barndorff-Nielsen and Shephard (2001). However, the different correlation structure of the well-balanced Ornstein-Uhlenbeck process is inherited by the integrated volatility and the squared price increments. 

In addition to the well-balanced Ornstein-Uhlenbeck process with the
kernel given above we also introduce the process with the
corresponding difference kernel $\exp(-\lambda
|t-\cdot|)-\exp(-\lambda |\cdot|)$, motivated by the form of the
kernel of fractional Brownian motion. This process, in contrast to
the previous one, is not stationary, but it possesses stationary
increments and starts in zero. Furthermore, the distribution of the
squared increments of both processes are obviously equal and the
autocorrelation function has the same decay.

Let us give a brief outline on how the paper is organized: in
Section 2 we introduce the notation and define the processes, in
Section 3 we show that both processes are semimartingales and derive
the structure of their characteristics. In Section 4 we provide the
moments and correlation structure of the processes. In Section 5 we
give a brief empirical example to high frequency
data. In Section 6 we indicate how the well-balanced Ornstein-Uhlenbeck process might be used as volatility process.

\section{Definition of the well-balanced Ornstein-Uh\-len\-beck process}
As driving process we consider a L\'{e}vy process $L$ given by the characteristic function $E(\exp(iuL_t))=\exp(t\psi(u))$ with
\[\psi(u)=iu\gamma-\sigma^2\frac{u^2}{2}+\int_{-\infty}^\infty \Big(\exp(iux)-1-iux1_{|x|\leq 1}\Big)\, \nu(dx),\]
where the L\'{e}vy measure $\nu$ satisfies the integrability
condition $\int_{-\infty}^{\infty}1\wedge x^2\, \nu(dx)<\infty$.

In the following we give conditions on a kernel function $f(\cdot,\cdot): \bbr^+_0\times \bbr \to \bbr^+_0$ such that processes of the form
\begin{eqnarray*}
Z_t=\int_{-\infty}^\infty f(t,s) \, dL_s,\quad t\geq 0
\end{eqnarray*}
exist. Here $L$ denotes the two-sided version of the L\'{e}vy
process which is defined in the straight forward way by taking two
independent copies $L^{(1)}$ and $L^{(2)}$ and defining
\[
L_t:=\begin{cases}\phantom{-}L^{(1)}_t & \text{ if }t\geq 0 \\
                  -L^{(2)}_{-t-} & \text{ if } t<0. \end{cases}
\]

Here and in the following we deal with stochastic integrals on the
real line as well as on the positive half line. Integrals on $\bbr$
are meant in the sense of Rajput and Rosinski (1989), i.e. we
associate an independently scattered random measure $\Lambda$ with
the two-sided L\'evy process $L$. For details we refer the reader to
Sato (2004) (in particular Theorem 3.2) who even treats the more
general case of additive processes in law on $[0,\infty)$. The
extension to $\bbr$ is straightforward. $\Lambda$ is defined on the
$\delta$-ring of bounded Borel measurable sets in $\bbr$ and the
integral $\int_\bbr g(s) \, d\Lambda_s$ is introduced in a canonical
way for deterministic step functions $g$. A function $f$ is then
called integrable if there exists a sequence $(g_n)_{n\in\bbn}$ of
step functions such that
\begin{itemize}
  \item $g_n \to f$ a.s with respect to the Lebesgue measure
  \item $\lim_{n\to\infty} \int_A g_n(s) \, d\Lambda_s$ exists for
  every $A\in\cb(\bbr)$.
\end{itemize}
If a function $f$ is integrable, we write $\int_\bbr f \,
dL_s=\lim_{n\to\infty} \int_\bbr g_n(s) \, d\Lambda_s$. From time to
time we will switch between this integral and the classical It\^o
integral, namely in the case $\int_\bbr 1_{[0,t]} f(s) \,
dL_s=\int_0^t f(s) \, dL_s$ for $f\in C_b$. Both integrals coincide
for predictable integrands of the type $f(s)=1_{[0,s)}$. The general
case follows by a standard argument using dominated convergence.
Before we specify the function $f(t,s)$ let us first briefly look at
the setting of a general kernel. Rewriting the criteria of Rajput
and Rosinski (1989) [Theorem 2.7] for the existence of the integral
we obtain: the stochastic integral $\int_\bbr f \, dL_s$ is well
defined if for $t\geq 0$
\begin{gather*}
  \int_{-\infty}^\infty\int_{-\infty}^\infty \abs{xf(t,s)}^2\wedge 1 \, \nu(dx) \, ds<\infty\\
  \int_{-\infty}^\infty\sigma^2 f(t,s)^2 \, ds<\infty\\
  \int_{-\infty}^\infty \Big|f(t,s)\Big(\gamma+\int_{-\infty}^\infty
    x\big(1_{|xf(t,s)|\leq 1}-1_{|x|\leq 1}\big) \, \nu(dx) \Big) \Big| \, ds<\infty
\end{gather*}
(cf. in this context Basse and Pedersen (2009) equations
(2.1)-(2.3)). Then the characteristic function is given by
\[
  E(\exp(iuZ_t))=\exp\left(\int\psi\big(uf(t,s)\big)\, ds\right)
\]
and $Z_t$ is infinitely divisible with characteristic triplet $(\gamma_f, \sigma_f^2,\nu_f)$
\begin{eqnarray*}
\gamma_f&=&\int_{-\infty}^\infty f(t,s)\left(\gamma+\int_{-\infty}^\infty x\big(1_{|xf(t,s)|\leq 1}-1_{|x|\leq 1}\big) \, \nu(dx)\right)ds\\
 \sigma_f^2&=&\int_{-\infty}^\infty\sigma^2 f(t,s)^2 \, ds\\
\nu_f(A)&=&(\nu\times\pmb{\lambda})\Big\{(x,s)\Big|xf(t,s)\in
A\setminus \{0\}\Big\},\quad A\in {\cal{B}}
\end{eqnarray*}
where $\pmb{\lambda}$ denotes Lebesgue measure. Furthermore for
$u_1,u_2,\cdots, u_m\in \bbr$ and
$-\infty<t_1<t_2<\cdots<t_m<\infty$ we obtain
\[
  E\Big(\exp(\sum_{j=1}^miu_j Z_{t_j})\Big)=\exp\Big(\int\psi\big(\sum_{j=1}^mu_jf(t_j,s)\big) \, ds\Big).
\]
If we now consider kernels of the form $f(t-s)$ the resulting
process $Z$ is stationary since for every $h\geq 0$
\begin{align*}
   E\Big(\exp(\sum_{j=1}^m iu_j Z_{t_j+h})\Big) &= \exp\Big(\int\psi\big(\sum_{j=1}^mu_jf(t_j+h-s)\big) \,
   ds\Big) \\
   &= \exp\Big(\int\psi\big(\sum_{j=1}^mu_jf(t_j-s)\big) \,
   ds\Big)\\
   &=E\Big(\exp(\sum_{j=1}^m iu_j Z_{t_j})\Big).
\end{align*}
In particular it possesses stationary increments.
If we consider kernels of the form $f(t-s)-f(s)$ we have $Z_0=0$ a.s. and stationary increments where the increments have the same distribution as the increments of the process generated by the kernel $f(t-s)$.

If $f(t,.)\in L^2(\bbr)$ and the second moment of $L$ exists and the
first one vanishes, we denote $E(L_1^2)=V$, then $Z_t$ also exists
in the $L^2$-sense with isometry 
\[
  EZ_t^2=||f(t,.)||^2_{L^2} V,
\]
as shown in Marquardt (2006)[Prop. 2.1].
Now we can come back to our special cases and assume $\lambda>0$.
For the well-balanced Ornstein-Uhlenbeck
process the kernel is
\[
  \exp(-\lambda |t-s|)=\exp\Big(-\lambda \big(\max(t-s,0)+\max(-(t-s),0)\big)\Big).
\]
From this reformulation we can see why we call the process
well-balanced Ornstein-Uhlenbeck process, namely
\[
  X_t=\int_{-\infty}^\infty \exp(-\lambda |t-s|) \, dL_s=\int_{-\infty}^t \exp(-\lambda (t-s)) \, dL_s+ \int^{\infty}_t \exp(-\lambda (s-t)) \, dL_s
\]
which is analogous to the well-balanced fractional L\'{e}vy motion
(cf.  Marquardt
(2006) [Definition 3.1]). Following the terminology of Samorodnitsky and Taqqu (1994) [Example 3.6.4] in the stable case it is the sum of an Ornstein-Uhlenbeck process and a reverse (or fully anticipating) Ornstein-Uhlenbeck process. The initial distribution of $X$ is given
by
\[
 X_0=\int_{-\infty}^\infty e^{-\lambda \abs{s}} \, dL_s= \int_{-\infty}^0 e^{\lambda s} \, dL_s + \int_0^\infty e^{-\lambda s} \, dL_s.
\]
As for the fractional kernel we can construct processes with stationary increments starting from zero, which for the well-balanced Ornstein-Uhlenbeck process leads to 
\[
  Y_t =X_t - X_0= \int_{-\infty}^{\infty}\exp(-\lambda |t-s|)-\exp(-\lambda |s|) \, dL_s.
\]
Now we can provide the characteristic triplet of the process $X$.
\begin{lemma}
The well-balanced Ornstein-Uhlenbeck process
\[
  X_t=\int_{-\infty}^{\infty} \exp(-\lambda |t-s|) \, dL_s
\]
is well-defined and infinitely divisible with characteristic triplet $(\gamma_X, \sigma_X^2, \nu_X)$
\begin{eqnarray*}
  \gamma_X  &=& \frac{2}{\lambda}\gamma+\frac{2}{\lambda}\Big(\int_1^\infty  \nu(dx)-\int_{-\infty}^{-1} \nu(dx)\Big)\\
  \sigma_X^2&=& \frac{1}{\lambda}\sigma^2\\
  \nu_X(A)  &=& (\nu\times \pmb{\lambda})\Big\{(x,s)\Big|x\exp(-\lambda |t-s|)\in A\setminus \{0\}\Big\},\quad A\in {\cal{B}}
\end{eqnarray*}
if and only if $\lambda>0$ and $\int x^2\wedge \log|x| \nu(dx)<\infty$.
\end{lemma}
\begin{proof} The result follows by straight forward calculations from the general formulae.
\end{proof}
Here we see that in contrast to fractional L\'{e}vy motions the
well-balanced Ornstein-Uhlenbeck process is well-defined with only
imposing the condition of a logarithmic moment of the driving L\'{e}vy process.
\section{Semimartingale Property and Characteristics}

Since the processes $X$ and $Y$ differ only by a random variable
which does not depend on $t\geq 0$, in the following we only treat
$X$. However, the results remain valid for $Y$.

We will show that $(X_t)_{t\geq 0}$ is a process of finite variation
on compacts and hence a semimartingale with respect to any
filtration it is adapted to. In order to do this we introduce the
following decomposition
\begin{align*}
\int_{-\infty}^\infty e^{-\lambda \abs{t-s}} \, dL_s
  &=e^{-\lambda t} \int_{-\infty}^0 e^{\lambda s} \, dL_s +  e^{-\lambda t} \int_0^t e^{\lambda s} \, dL_s + e^{\lambda t} \int_t^\infty e^{-\lambda s} \, dL_s
\end{align*}
and write the last term as
\[
e^{\lambda t} \int_t^\infty e^{-\lambda s} \, dL_s =e^{\lambda t} \int_0^\infty e^{-\lambda s} \, dL_s
  - e^{\lambda t}\int_0^t e^{-\lambda s} \, dL_s.
\]
For reference purposes we write the above representation of $(X_t)_{t\geq 0}$ in a short form
\begin{align} \label{repr}
X_t= e^{-\lambda t} G + e^{\lambda t} H  + e^{-\lambda t} I_t - e^{\lambda t} J_t,
\end{align}
using the following notation:
\[
  I_t:=\int_0^t e^{\lambda s} \, dL_s \text{ and } J_t:=\int_0^t e^{-\lambda s} \, dL_s
\]
and
\[
  G:=\int_{-\infty}^0 e^{\lambda s} \, dL_s \text{ and } H:=\int_0^\infty e^{-\lambda s} \, dL_s.
\]

From \eqref{repr} above and I.4.36 in Jacod and Shiryaev (2003) we
obtain:
\[
\Delta X_t = \Delta(e^{-\lambda t} I_t - e^{\lambda t}
J_t)=e^{-\lambda t}(e^{\lambda t} \Delta L_t)-e^{\lambda
t}(e^{-\lambda t} \Delta L_t)=0
\]
for every  $t\geq 0$. This means the process is continuous, though
the driving term exhibits jumps. The following theorem refines this
result.

\begin{theorem}
The process $(X_t)_{t\geq 0}$ is of finite variation on compacts and
hence it is a semimartingale with respect to \emph{any} filtration
it is adapted to. Furthermore the process is locally Lipschitz
continuous.
\end{theorem}

\begin{proof}
First we use the integration-by-parts formula on $I_t$ and $-J_t$
and obtain
\[
\int_0^t e^{\lambda s} \, dL_s = - \int_0^t L_s   \lambda e^{\lambda
s} \, ds + L_t  e^{\lambda t}
\]
respective
\[
-\int_0^t e^{-\lambda s} \, dL_s = - \int_0^t L_s   \lambda
e^{-\lambda s} \, ds - L_t  e^{-\lambda t}.
\]
Putting these together we have
\begin{align*}
&X_t-X_0=Y_t\\
&=(e^{-\lambda t}-1)\int_{-\infty}^0 e^{\lambda s} \, dL_s + (e^{\lambda t}-1) \int_0^\infty e^{-\lambda s} \, dL_s- e^{-\lambda t}\int_0^t L_s   \lambda e^{\lambda s} \, ds - e^{\lambda t} \int_0^t L_s   \lambda e^{-\lambda s} \, ds \\
&= \int_0^t \Bigg\{- \left(\int_{-\infty}^0 e^{\lambda u} \, dL_u \right) \lambda e^{-\lambda s}  + \left(\int_0^\infty e^{-\lambda u} \, dL_u \right) \lambda e^{\lambda s} \\
&\hspace{5mm}+\lambda e^{-\lambda s} \int_0^s L_r \lambda e^{\lambda r} \, dr - e^{-\lambda s}   L_s \lambda e^{\lambda s} - \lambda e^{\lambda s} \int_0^s L_r \lambda e^{-\lambda r} \, dr - e^{\lambda s}   L_s \lambda e^{-\lambda s}  \Bigg\} \, ds \\
&= \int_0^t \Bigg\{- G \lambda e^{-\lambda s}  + H \lambda e^{\lambda s} - 2 L_s \lambda +\lambda e^{-\lambda s} \int_0^s L_r \lambda e^{\lambda r} \, dr   - \lambda e^{\lambda s} \int_0^s L_r \lambda e^{-\lambda r} \, dr  \Bigg\} \, ds\\
&=\int_0^t \lambda \Bigg\{ -Ge^{-\lambda s} + H e^{\lambda s} -
 I_s e^{-\lambda s} - J_s e^{\lambda s} 
\Bigg\} \, ds\\
&=\int_0^t \lambda \Big( e^{\lambda s} \int_s^\infty
e^{-\lambda r} \, dL_r - e^{-\lambda s} \int_{-\infty}^s e^{\lambda
r} \, dL_r\Big) \, ds.
\end{align*}
This means we end up with a pathwise integration of functions which
are a.s. locally bounded since they are c\`adl\`ag. In particular
the process is locally Lipschitz and of finite variation on
compacts.
\end{proof}
Note that this is different to the classical Ornstein-Uhlenbeck process
which inherits the jump property from the driving process.

In Basse and Pedersen (2009) and Bender et.al (2010) the authors
treat the case of other kernel functions. However, their conditions
on the L\'evy process are more restrictive. Note that by Jacod and
Shiryaev (2003) Proposition I.4.24 the process $(X_t)_{t\geq 0}$,
since it does not have jumps, is even a special semimartingale with
respect to every filtration it is adapted to. For the remainder of
the paper we fix the filtration $\bbf$ which is obtained by defining
first $\bbf^1=(\cf_t^1)_{t\geq 0}$ via
$\cf_t^1:=\sigma(\cf_t^0,G,H)$. Which is completed and made right
continuous in the usual way to obtain $\bbf$.

Since the process is continuous and of finite variation on compacts
we obtain the following corollary by the general representation of
semimartingales (cf. Jacod and Shiryaev (2003) Theorem II.2.34).

\begin{corollary}
The semimartingale characteristics of the process $X$ are
$(X_t-X_0,0,0)$.
\end{corollary}

A different approach leads to another perspective on the
well-balanced Ornstein-Uhlenbeck process. The process $X$ can be
represented as an ARMA process in continuous time, i.e. a CARMA
process (cf. Brockwell and Lindner (2009)). In general such a
process, possibly complex valued, is given by
\[
X_t=b'R_t
\]
where $b\in\bbc^p$ and the $\bbc^p$-valued stochastic process $R$ is
given as the solution of the SDE
\[
dR_t=\left( \begin{array}{ccccc} 0 & 1 & 0 & \cdots & 0 \\
                                 0 & 0 & 1 & \cdots & 0 \\
                                 \vdots &  \vdots &  \vdots & \ddots
                                 & \vdots \\
                                 0 & 0 & 0 & \cdots & 1 \\
                                 -a_p & -a_{p-1} & -a_{p-2} &
                                 \cdots& -a_1
\end{array}\right) R_t \, dt + \left(\begin{array}{c} 0 \\ 0 \\ \vdots \\
0 \\ 1 \end{array}\right) \, dL_t
\]
where $a_1,...,a_p$ are complex-valued coefficients and $L$ is a
driving L\'evy process. The $p\times p$-matrix is usually denoted by
$A$. In our case we have
\[
p=2,\, A=\left(\begin{array}{cc} 0 & 1 \\ \lambda^2 & 0 \end{array}
\right) , \, b=\binom{-2\lambda}{0}.
\]
The SDE for the well-balanced L\'evy driven Ornstein-Uhlenbeck
process is
\begin{align*}
dR_t^{(1)} &= R_t^{(2)} \, dt +0 \\
dR_t^{(2)} &=\lambda^2 R_t^{(1)} \, dt + dL_t
\end{align*}
with the explicit solution
\[
R_t=e^{At}R_0 + \int_0^t e^{A(t-u)} \binom{0}{1} \, dL_u.
\]
Here $e^{At}$ is not meant componentwise but in the sense of bounded operators, i.e. as a convergent series of matrices.
By choosing the initial conditions appropriately, i.e.
\[
R_0^{(1)}:= -\frac{1}{\lambda} \left( \frac{G+H}{2} \right); \,
R_0^{(2)}:= \left( \frac{G-H}{2} \right)
\]
we get our process $X$ as a stationary CARMA(2,0) process. For
criteria for CARMA processes to be stationary compare Brockwell and
Lindner (2009) [Theorem 3.3].

Summarizing we can see that though integrating with respect to a
quite general L\'{e}vy process the special very
regular form of the kernel leads to a semimartingale of bounded
variation. Hence the well-balanced Ornstein-Uhlenbeck process might
serve as mean process in the framework of semimartingale models,
e.g. stochastic volatility models in finance.

\section{Moments and Correlation Structure}
In this section we will analyze the correlation structure of the
well-balanced Ornstein-Uhlenbeck process. We will see that though
the process is closely related to the stationary version of an
Ornstein-Uhlenbeck process the two-sided kernel leads to a different
behaviour in the autocorrelation function, namely to a decay of the order $\lambda |h|\exp(-\lambda |h|)$,
 and to a bigger range of possible values as in classical Ornstein-Uhlenbeck
process, including positive
ones, in the first order autocorrelation of increments.
\begin{proposition}
Let $X_t=\int \exp(-\lambda|t-u|)dL_u$ and assume that the driving
L\'{e}vy process possesses a finite second moment. We denote the variance of $L(1)$ by $V$ and the first moment by $\mu$, then we
obtain the following characteristic quantities for $X$ and $h\geq 0$
\begin{eqnarray*}
 EX_t&=&\frac{2\mu}{\lambda}\\
 var(X_t)&=&\frac{V}{\lambda}\\
 Cov(X_{t+h}, X_t)&=&Vhe^{-\lambda h}+\frac{V}{\lambda}e^{-\lambda h}\\
Corr(X_{t+h},X_t)&=& \lambda h e^{-\lambda h}+ e^{-\lambda h}.
\end{eqnarray*}
\end{proposition}
\begin{proof}
From the general form of the characteristic function, we can calculate the second moment of $Z_t=\int f(t,s)dL_s$, provided that $L$ possesses a second moment and both $f(t,.)$ and $f(t,.)^2$ are integrable. We obtain
\begin{eqnarray*}
EZ_t&=&\int f(t,s)ds\left(\gamma+\int x1_{|x|>1}\nu(dx)\right)\\
EZ_t^2&=&\int f(t,s)^2ds \left(\sigma^2+\int
x^2\nu(dx)\right)+\left(\int f(t,s)ds\right)^2\left(\gamma+\int
x1_{|x|>1}\nu(dx)\right)^2.
\end{eqnarray*}
In the following we denote $\sigma^2+\int x^2\nu(dx)=V$ and $\gamma+\int x1_{|x|>1}\nu(dx)=\mu$. Using this together with the independent increment property of $L$, we obtain for $X_t=\int \exp(-\lambda|t-u|)dL_u$ and $s\leq t$
\begin{eqnarray*}
EX_t&=&\frac{2\mu}{\lambda}\\
EX_t^2&=&\frac{V}{\lambda}+\frac{4\mu^2}{\lambda^2}\\
var(X_t)&=&\frac{V}{\lambda}\\
 E(X_t-X_s)^2&=&\frac{2V}{\lambda}\big(1-e^{-\lambda(t-s)}-\lambda(t-s)e^{-\lambda(t-s)}\big).
\end{eqnarray*}

Hence
\begin{eqnarray*}
Cov(X_t,X_s)&=&\frac{1}{2}\big(EX_t^2-E(X_t-X_s)^2+EX_s^2\big)-EX_tEX_s\\
&=&V(t-s)e^{-\lambda(t-s)}+\frac{V}{\lambda}e^{-\lambda(t-s)}\\
Corr(X_t,X_s)&=& \lambda(t-s)e^{-\lambda(t-s)}+ e^{-\lambda(t-s)}.
\end{eqnarray*}
\end{proof}
Comparing this to the well known quantities of a stationary Ornstein-Uhlenbeck process $U$
we see, while the mean and the variance only differ by a multiple of two, the auto-covariance and autocorrelation function have an extra term leading to a slower decay, when the same $\lambda$ is considered. This might be an interesting feature for modelling data, especially coming from finance, where a pure exponential decay often seem too fast to match the empirical autocorrelation properly.

If we define the well-balanced Ornstein-Uhlenbeck process with  time scaled by $\lambda$, i.e. $X_t=\int\exp(-\lambda|t-u|)dL_{\lambda u}$ this has the same effect as in Barndorff-Nielsen and Shephard (2001), that the marginal distribution and hence the moments are independent of $\lambda$.

Also the correlation between increments might be of interest for
modelling purposes and follows by direct calculations from the
proposition above.
\begin{corollary}
Assume the same conditions on $L$ as in the previous proposition, then we obtain
\begin{eqnarray*}
Corr(X_{k+1}-X_k, X_1-X_0)&=&\exp(-\lambda k)\left(\frac{1}{2}+\frac{1}{2}\frac{1-\exp(\lambda)+\lambda \exp(\lambda)}{1-\exp(-\lambda)-\lambda \exp(-\lambda)}\right)\\
&&+\lambda k \exp(-\lambda
k)\left(\frac{1}{2}+\frac{1}{2}\frac{1-\exp(\lambda)+\lambda
\exp(-\lambda)}{1-\exp(-\lambda)-\lambda \exp(-\lambda)}\right)
\end{eqnarray*}
and as a special case the first-order autocorrelation
\[
  Corr(X_{2}-X_1, X_1-X_0)=\exp(-\lambda )\left(\frac{1+\lambda}{2}+\frac{1}{2}\frac{1+\lambda-\exp(\lambda)+\lambda^2 \exp(-\lambda)}{1-\exp(-\lambda)-\lambda \exp(-\lambda)}\right).
\]

\end{corollary}
Note that in contrast to the classical Ornstein-Uhlenbeck process
whose autocorrelation function of increments $Corr(U_{k+1}-U_k,
U_1-U_0)=\exp(-\lambda
k)(\frac{1}{2}+\frac{1}{2}\frac{1-\exp(\lambda)}{1-\exp(-\lambda)})$
is always negative in the range between -0.5 and 0, we can have
positive and negative values, in the range from -0.5 to 1 depending
on $\lambda$ for the well-balanced Ornstein-Uhlenbeck process.
Looking for example at the first-order autocorrelation
$Corr(X_{2}-X_1, X_1-X_0)$ it is positive for $\lambda < 1.25643$
and negative for bigger values of $\lambda$. This provides much more
flexibility for modelling, e.g. we can obtain values for the
first-order autocorrelation which is often linked to long-range
dependence. Assuming that $B_t^H$ denotes a fractional Brownian
motion with Hurst parameter $H\in(0,1)$, then by Kettani and Gubner
(2006)
\[\frac{\sum_{i=1}^{n-1}(X_{i}^H-\bar{X}_n^H)(X_{i+1}-\bar{X}_n^H)}{\sum_{i=1}^{n}(X_{i}^H-\bar{X}_n^H)^2}\to C_H=2^{2H-1}-1,\]
where $X_i^H=B_i^H-B_{i-1}^H$ and $\bar{X}_n^H=\frac{1}{n}\sum_{i=1}^nX_i^H$.
Hence we can see that as the first-order autocorrelation of the well-balanced Ornstein-Uhlenbeck process $C_H\in (-0.5, 1)$ and $C_H>0$ for $H>0.5$.

For some applications it might of course be more realistic not to
have a stationary process, but a process with stationary increments
like L\'{e}vy processes. In the context of well-balanced
Ornstein-Uhlenbeck processes we can construct processes with the
same correlation structure of increments and the same paths
regularity by considering the associated difference kernel.
\begin{proposition}
Let $Y_t=\int_{-\infty}^\infty \exp(-\lambda |t-s|)-\exp(-\lambda |s|)dL_s$ and assume that the driving L\'{e}vy process possesses a finite second moment. We denote the variance by $V$ and the first moment by $\mu$, then we obtain the following characteristic quantities for $Y$
\begin{eqnarray*}
 EY_t&=&0\\
 var(Y_t)&=&Vte^{-\lambda t}+\frac{V}{\lambda}e^{-\lambda t}\\
Corr(Y_{k+1}-Y_k, Y_1-Y_0)&=&\exp(-\lambda k)\left(\frac{1}{2}+\frac{1}{2}\frac{1-\exp(\lambda)+\lambda \exp(\lambda)}{1-\exp(-\lambda)-\lambda \exp(-\lambda)}\right)\\
&&+\lambda k \exp(-\lambda
k)\left(\frac{1}{2}+\frac{1}{2}\frac{1-\exp(\lambda)+\lambda
\exp(-\lambda)}{1-\exp(-\lambda)-\lambda \exp(-\lambda)}\right).
\end{eqnarray*}
\end{proposition}
\begin{proof} The proof follows immediately by noting that $Y_t=X_t-X_0$.
\end{proof}
Note that we can easily also construct a process which only possess this correlation structure for a specific lag and is zero for larger lags.
For a kernel on a compact interval $[0,a]$ we obtain the process $X_t=\int_{t-a}^t\exp(-\lambda (t-s))dL_s$ which possesses the second moment $EX_t^2=(1-\exp(-2\lambda a))/(2\lambda)$, assuming for simplicity that the first moment of $L(1)$ is zero and the second 1. Furthermore for increments $X_t-X_s$ we obtain $E(X_t-X_s)^2=(1-\exp(-2\lambda a)-\exp(-\lambda (t-s))+\exp(-\lambda(2a+s-t)))/\lambda$ if $t-s\leq a$  and if $t-s>a$: $E(X_t-X_s)^2=EX_t^2+EX_s^2$. This leads to $Cov(X_t,X_s)=(\exp(-\lambda(t-s))-\exp(-\lambda(2a+s-t)))/(2\lambda)$ for $t-s\leq a$ and 0 otherwise.

\section{Application to high frequency data}
We apply the well-balanced
Ornstein-Uhlenbeck process to an example of real data and show that
the autocorrelation models the empirical autocorrelation quite well.
Hence this indeed offers the possibility of adding the well-balanced
Ornstein-Uhlenbeck process $X$ as an empirically convincing mean process
to a classical stochastic volatility model, i.e. for the log-price process $Y_t=X_t+\int_0^t\sigma_sdW_s+J_t$, where $\sigma$ denotes a volatility process, $W$ a Brownian motion and $J$ a jump L\'{e}vy process.

First we consider one trading day of the SAP share, namely of 1st February 2006 9:00 am to 5:30 pm consisting of 5441 trades.
\begin{center}
\includegraphics[width=60mm]{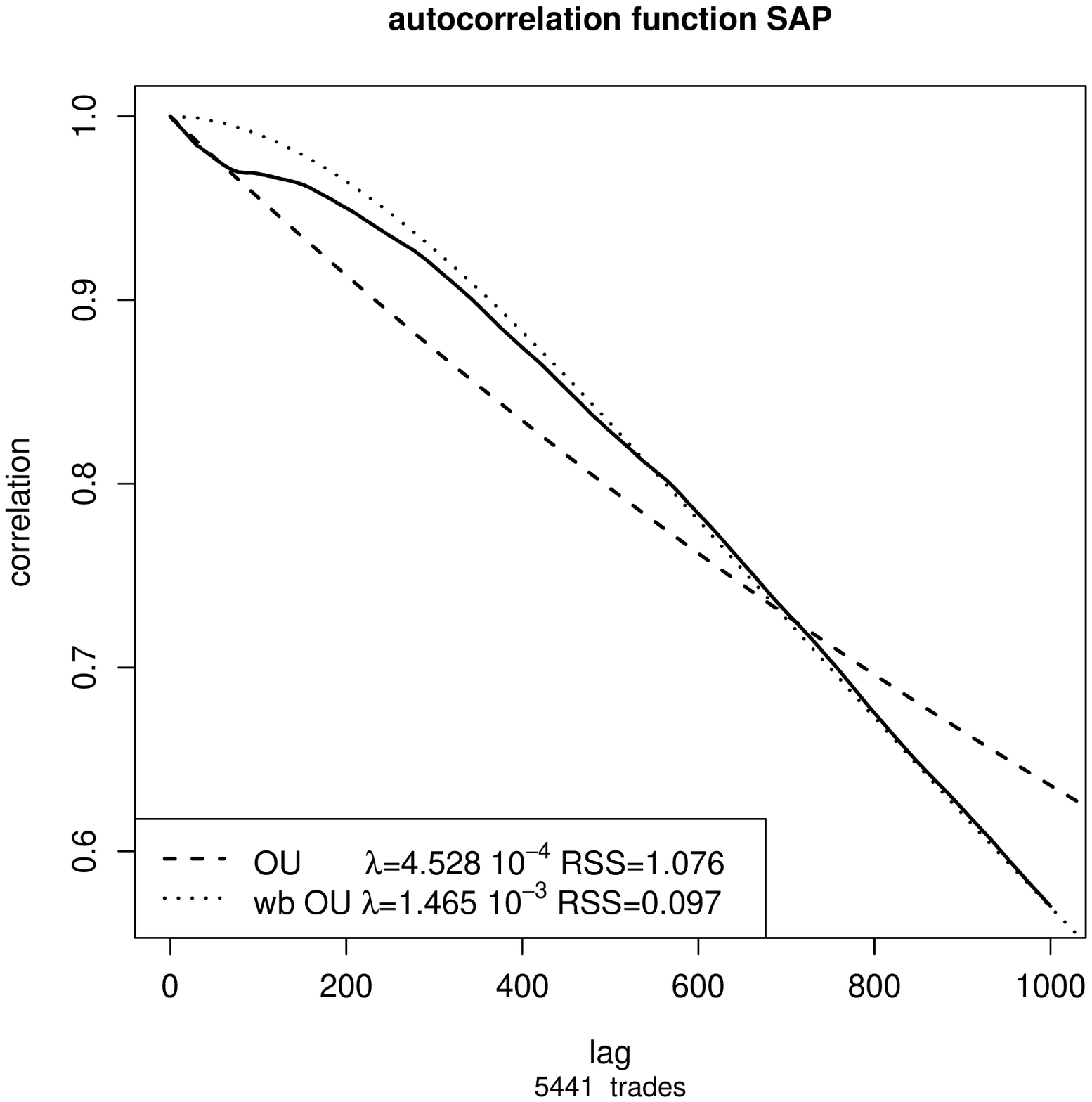}
\includegraphics[width=60mm]{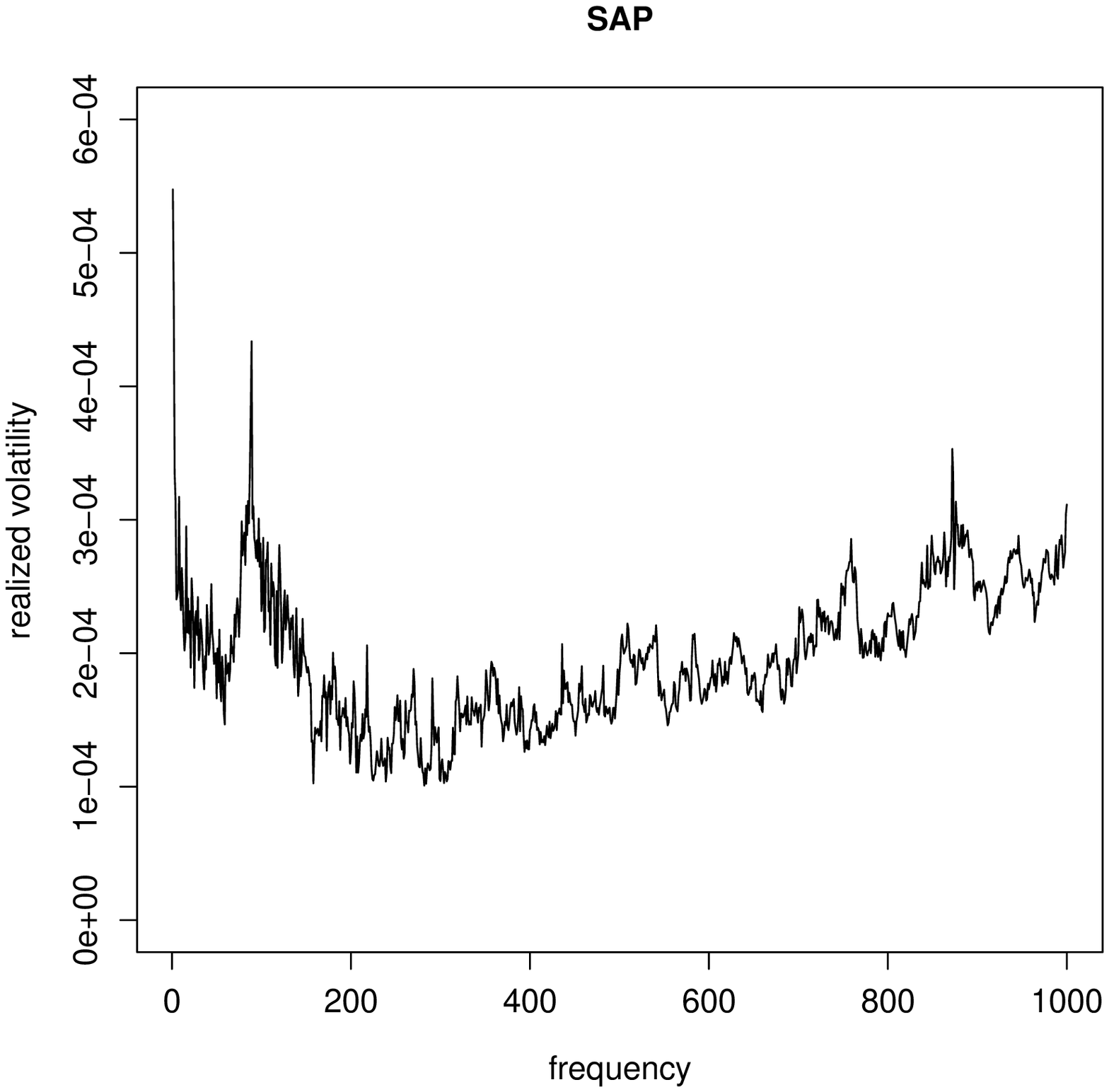}
\end{center}
The picture shows the empirical autocorrelation function as solid line, the dashed line is the fit with a classical Ornstein-Uhlenbeck process and the dotted line with the well-balanced Ornstein-Uhlenbeck process. We can see that the autocorrelation function both visually and by taking the residual sum of squares fits the data much better than the Ornstein-Uhlenbeck process, except for small lags. This might be interpreted as the effects of market microstructure. Namely the two kinks in the empirical curve are at a lag of 75 and 150 respectively.  In this setting this correspond to a sampling frequency of 7 minutes and 14 minutes. Values in this range are in the econometrics literature often seen as sampling frequencies from which market microstructure effects start to be negligible. We have also included the volatility signature plot for the one trading day of 1st February 2006, i.e. the realized volatility $\sum_{i=1}^n(X_{i/n}-X_{(i-1)/n})^2$ plotted with decreasing sampling frequency. It shows the presence of market microstructure when the realized volatility increases with increasing sampling frequency. For the SAP data we can see indeed this effect, which vanishes above roughly 150, but possesses an additional peak at about 80.
If we start fitting the empirical data only for larger lags than 150, the values of $\lambda$ and the RSS for the Ornstein-Uhlenbeck process stay the same, whereas the RSS of the well-balanced Ornstein-Uhlenbeck processes decreases to $4.393 10^{-3}$.

Furthermore, we also examined different data sets, namely tick-by-tick data of the Daimler Chrysler share and the Siemens share of January 2005. Fitting the empirical autocorrelation with the autocorrelation function of the well-balanced Ornstein-Uhlenbeck process and the Ornstein-Uhlenbeck process up to a lag of 1000, the well-balanced Ornstein-Uhlenbeck did better on 17 of the 21 trading days for Daimler Chrysler and on 13 for Siemens. Hence for single days the well-balanced Ornstein-Uhlenbeck process provides a good alternative. For averaged data the situation is a bit different. When averaging of the trading days the well-balanced Ornstein-Uhlenbeck process performed better than the Ornstein-Uhlenbeck process for Daimler Chrysler and worse for Siemens. Especially for small lags the Ornstein-Uhlenbeck process seems to fit very well as averaging smoothes away the kinks in the empirical curves of the single days.

\section{Well-balanced Ornstein-Uhlenbeck Process as Volatility Process}
In this section we examine how the well-balanced Ornstein-Uhlenbeck
process might serve as volatility process in a stochastic volatility
model along the lines of the Non-Gaussian Ornstein-Uhlenbeck models
of Barndorff-Nielsen and Shephard (2001). The well-balanced
Ornstein-Uhlenbeck process leads to a model which is as tractable as
the Ornstein-Uhlenbeck model but possesses some new features, namely
that the different decay of the autocorrelation function is
inherited by the decay of the autocorrelation of increments of
integrated volatility and the autocorrelation of squared
log-returns. 

As in Barndorff-Nielsen and Shephard (2001) we consider the stochastic volatility model of the type
\[ dY_t=(\alpha+\beta X_t)dt+X_t^{1/2}dW_t.\]
We assume $\alpha,\beta\in \bbr$, $W$ denotes a Brownian motion and
$X$ the spot volatility process given by
\[ X_t=\int_{-\infty}^\infty e^{-\lambda|t-u|}dL_{\lambda u},\]
with $\lambda>0$ and $L$ a two-sided L\'{e}vy process. We assume that the law of $L(1)$ is self-decomposable and restricted to the positive half line and the cumulant generating function is given by
\[k(\theta)=\log(E(\exp(-\theta L(1))))=-\int_{0+}^\infty 1-\exp(-\theta x)\nu(dx),\]
where $\nu$ denotes the L\'{e}vy measure and the first two moments of $L(1)$
exist.  We denote the first moment by $\mu$ and the variance by $V$. The time scale $\lambda t$ in the definition of $X$
ensures that the marginal law is independent of $\lambda$ and the
support of $\nu$ that $X$ is non-negative as required for a
volatility process.

Similarly as for the characteristic function in Section 2 also the
cumulant transform $\bar{k}$ of $X$ can be expressed in terms of the
kernel and the cumulant transform of the underlying L\'{e}vy
process, namely
\[\bar{k}(\theta)=\int_{-\infty}^\infty k(\theta e^{-|\lambda t -s|})ds=2\int_0^\infty k(\theta e^{-u})du.\]
From this we obtain that the n-th cumulant $m_n$ of the underlying
L\'{e}vy process is related to the n-th cumulant  $\bar{m}_n$ of $X$
by $m_n=\frac{n}{2}\bar{m}_n$. Furthermore, if $\nu$ possesses a
Lebesgue density $g$, also $\nu_{e^{-|.|}}$ possesses a Lebesgue
density $\bar{g}(y)=2\int_1^\infty g(xy)dx$ or respectively
$g(y)=\frac{1}{2}(-\bar{g}(y)-y\bar{g}'(y))$. This implies that the
results differ only by 2 and 0.5 from the formulae of the
Ornstein-Uhlenbeck processes respectively, hence modelling of the marginals is
the same.

Next by direct calculation we can provide an explicit formula for the integrated volatility
\[\int_0^tX_u du=\frac{1}{\lambda}(2L_{\lambda t}+X_0^--X_0^+-(X_t^--X_t^+)),\]
where $X_t=\int_{-\infty}^t e^{-\lambda(t-u)}dL_{\lambda u}+\int_{t}^\infty e^{-\lambda (u-t)}dL_{\lambda u}=X_t^-+X_t^+$. Hence we can see that we have the additional terms $\frac{1}{\lambda}(L_{\lambda t}-X_0^++X_t^+)$ compared to the formula of the Ornstein-Uhlenbeck process which is provided in Barndorff-Nielsen and Shephard (2001).

Finally we would like to note that the correlation structure of $X$ is inherited by the integrated volatility and the squared returns. By the general formulae provided in Barndorff-Nielsen and Shephard (2001) we obtain for $n,s\geq 1$ setting $\alpha=\beta=0$
\begin{eqnarray*}
Cov\left(\int_{(n-1)\Delta}^{n\Delta} X_udu,\int_{(n+s-1)\Delta}^{(n+s)\Delta} X_udu\right)&=&VR(\Delta s)\\
Corr((Y_{n\Delta}-Y_{(n-1)\Delta})^2,(Y_{(n+s)\Delta}-Y_{(n+s-1)\Delta})^2)&=& \frac{R(\Delta s)}{6\bar{r}(\Delta)+2\Delta^2\frac{\mu^2}{V}}
\end{eqnarray*}
where for $t\geq 0$
\begin{eqnarray*}
r(t)&=&\lambda t e^{-\lambda t}+e^{-\lambda t}\\
\bar{r}(t)&=&\int_0^t\int_0^ur(x)dxdu=\frac{1}{\lambda^2}(\lambda t e^{-\lambda t}+2\lambda t+3e^{-\lambda t}-3)\\
R(\Delta s)&=&\bar{r}(\Delta(s+1))-2\bar{r}(\Delta s)+\bar{r}(\Delta(s-1))\\
&=&\frac{1}{\lambda^2}e^{-\lambda\Delta s}[(\lambda\Delta s+3)(e^{-\lambda\Delta}+e^{\lambda \Delta}-2)+\lambda \Delta(e^{-\lambda\Delta}-e^{\lambda\Delta})].
\end{eqnarray*}

{\bf Acknowledgements.} The authors thank Alexander D\"urre for programming the empirical example, Jan Kallsen and Alexander Lindner for stimulating discussions and the two anonymous referees for helpful comments and suggestions. The financial support of the Deutsche Forschungsgemeinschaft (SFB 823: Statistical modelling of nonlinear dynamic processes, project C5) is gratefully acknowledged.


\end{document}